\documentclass[12pt]{amsart}
\usepackage{amsfonts}
\usepackage{amsthm}
\usepackage[centertags]{amsmath}
\usepackage{amssymb}
\usepackage{enumerate}
\usepackage{graphicx}
\usepackage[bookmarks=true, pdfstartview={FitH}, pdftitle={Paper},
pdfauthor={Sadia Arshad}]{hyperref}
\usepackage{color}
\usepackage{afterpage}
\usepackage{boxedminipage}
\usepackage{fancybox}
\usepackage{fancyhdr}
\usepackage{inputenc}

\providecommand{\U}[1]{\protect\rule{.1in}{.1in}}
\theoremstyle{plain}
\newtheorem{theorem}{Theorem}[section]
\newtheorem{corollary}[theorem]{Corollary}
\newtheorem{definition}[theorem]{Definition}
\newtheorem{example}[theorem]{Example}

\newtheorem{proposition}[theorem]{Proposition}
\newtheorem{remark}[theorem]{Remark}

\numberwithin{equation}{section}
\input{tcilatex}

\begin{document}

\begin{center}
{\Large A new concept of differentiability for interval-valued functions} 
{\Large \textbf{\ }}\\[5mm]
{\large {Vasile Lupulescu}}$^{1}${\large {\ and Donal O'Regan}}$^{2}${\large 
{\ }\\[10mm]
}

$^{1}${\footnotesize Constantin Brancusi University, Republicii 1, 210152
Targu-Jiu, Romania,}

{\footnotesize e-mail: lupulescu\_v@yahoo.com}

$^{2}${\footnotesize School of Mathematics, Statistics and Applied
Mathematics, National University of Ireland, Galway, Ireland.}

{\footnotesize e-mail: donal.oregan@nuigalway.ie}\\[5mm]
\end{center}

{\footnotesize \textbf{Abstract}. }In this paper we propose a new concept of
differentiability for interval-valued functions. This concept is based on
the properties of the Hausdorff-Pompeiu metric and avoids using the
generalized Hukuhara difference. 
\afterpage{
\fancyhead{} \fancyfoot{} 
\fancyhead[LE, RO]{\bf\thepage}
\fancyhead[LO]{\small Short title}
\fancyhead[RE]{\small Author(s) name  }
\fancyfoot[C]{\small******************************************************************************\\
Surveys~in~Mathematics~and~its~Applications
\href{http://www.utgjiu.ro/math/sma/v13/v13.html}{{\bf 13}~(2018)}, xx -- yy \\
 \href{http://www.utgjiu.ro/math/sma}{http://www.utgjiu.ro/math/sma}}}

\section{\protect\Large Preliminaries}

Let us denote by $\mathcal{K}$ the set of all nonempty compact intervals of
the real line $\mathbb{R}$. If $A=[a^{-},a^{+}]$, $B=[b^{-},b^{+}]\in 
\mathcal{K}$, then the usual interval operations, i.e. Minkowski addition
and scalar multiplication, are defined by%
\begin{equation*}
A+B=[a^{-},a^{+}]+[b^{-},b^{+}]=[a^{-}+b^{-},a^{+}+b^{+}]
\end{equation*}%
and%
\begin{equation*}
\lambda A=\lambda \lbrack a^{-},a^{+}]=\left\{ 
\begin{tabular}{ll}
$\lbrack \lambda a^{-},\lambda a^{+}]$ & if $\lambda >0$ \\ 
$\{0\}$ & if $\lambda =0$ \\ 
$\lbrack \lambda a^{+},\lambda a^{-}]$ & if $\lambda <0$,%
\end{tabular}%
\right. 
\end{equation*}%
respectively. If $\lambda =-1$, scalar multiplication gives the opposite%
\begin{equation*}
-A:=(-1)A=(-1)[a^{-},a^{+}]=[-a^{+},-a^{-}].
\end{equation*}%
In general, $A+(-A)\neq \{0\}$; that is, the opposite of $A$ is not the
inverse of $A$ with respect to the Minkowski addition (unless $A=\{a\}$ is a
singleton). Minkowski difference is $A-B=A+(-1)B=[a^{-}-b^{+},a^{+}-b^{-}]$.
\vskip0.3cm

\noindent The \emph{generalized Hukuhara difference }(or $gH$\emph{%
-difference}) of two intervals $[a^{-},a^{+}]$, $[b^{-},b^{+}]\in \mathcal{K}
$ is defined as follows (Markov\textbf{\ }\cite{mar}): 
\begin{equation}
\lbrack a^{-},a^{+}]\ominus \lbrack b^{-},b^{+}]=[\min
\{a^{-}-b^{-},a^{+}-b^{+}\},\max \{a^{-}-b^{-},a^{+}-b^{+}\}].  \label{gH}
\end{equation}

\noindent We denote the \emph{width} of an interval $A=[a^{-},a^{+}]$ by $%
w(A)=a^{+}-a^{-}$. Then, for $A=[a^{-},a^{+}]$ and $B=[b^{-},b^{+}]$, we have%
\begin{equation}
A\ominus B=\left\{ 
\begin{tabular}{l}
$\lbrack a^{-}-b^{-},a^{+}-b^{+}]$, if $w(A)\geq w(B)$ \\ 
$\lbrack a^{+}-b^{+},a^{-}-b^{-}]$, if $w(A)<w(B).$%
\end{tabular}%
\right.   \label{gH_1}
\end{equation}%
If $A,B,C\in \mathcal{K}$ then it is easy to see that 
\begin{equation}
A\ominus B=C\Longleftrightarrow \left\{ 
\begin{tabular}{l}
$A=B+C$, if $w(A)\geq w(B)$ \\ 
$B=A+(-C)$, if $w(A)<w(B).$%
\end{tabular}%
\right.   \label{equiv-1}
\end{equation}%
If $A,B\in \mathcal{K}$ and $w(A)\geq w(B)$,\ then the $gH$-difference $%
A\ominus B$ will be denoted by $A\ominus B$ and it is called the \emph{%
Hukuhara difference }(or $H$\emph{-difference}) of $A$ and $B$. For other
properties involving the operations on $\mathcal{K}$, see Markov \cite{mar}.
If $A\in \mathcal{K}$, let us define the norm of $A$ by$\left\Vert
A\right\Vert :=\max \{\left\vert a^{-}\right\vert ,\left\vert
a^{+}\right\vert \}$. Then it is easy to see that $\left\Vert \cdot
\right\Vert $ is a norm on $\mathcal{K}$, and therefore $(\mathcal{K}%
,+,\cdot ,\left\Vert \cdot \right\Vert )$ is a normed quasilinear space. A
metric structure on $\mathcal{K}$ is given by the Hausdorff-Pompeiu distance 
$\mathcal{H}:\mathcal{K}\times \mathcal{K}\rightarrow \lbrack 0,\infty 
\mathbb{)}$ defined by $\mathcal{H}(A,B)=\max \{\left\vert
a^{-}-b^{-}\right\vert ,\left\vert a^{+}-b^{+}\right\vert \}$, where $%
A=[a^{-},a^{+}]$ and $B=[b^{-},b^{+}]$. Obviously, the metric $\mathcal{H}$
is associated with the norm $\left\Vert \cdot \right\Vert $ by $\left\Vert
A\right\Vert =\mathcal{H}(A,\{0\})$ and $\mathcal{H}(A,B)=\left\Vert
A\ominus B\right\Vert $.\ It is well known that $(\mathcal{K},\mathcal{H})$
is a complete. \vskip0.3cm

\begin{proposition}
\label{prop-3} \textit{The }$H$\textit{-difference }$\ominus $\textit{\ has
the following properties }(see \cite{lak1}):\newline
(a) $A\ominus \theta =A$\ and $A\ominus A=\theta $\textit{\ for all} $A\in 
\mathcal{K}$.\newline
(b) $(-A)\ominus (-B)=-(A\ominus B)$\textit{. }\newline
(c)\textit{\ }$(A+B)\ominus B=A$.\newline
(d) $A+(B\ominus A)=B$.\newline
(e) $(B+C)\ominus A=B\ominus A+C$.\newline
(f) $A\ominus B+C\ominus D=(A+C)\ominus (B+D)$.\newline
(g) $A\ominus B+B\ominus C=A\ominus C$.\newline
\end{proposition}

\vskip0.3cm

\begin{proposition}
\label{p-1} The Hausdorff-Pompeiu distance has the following properties (see 
\cite{lak1}):

(i) $\mathcal{H}(A+C,B+C)=\mathcal{H}(A,B),$\newline
(ii) $\mathcal{H}(\lambda A,\lambda B)=|\lambda |\mathcal{H}(A,B)$ for $%
\lambda \in \mathbb{R},$\newline
(iii) $\mathcal{H}(A+B,C+D)\leq \mathcal{H}(A,C)+\mathcal{H}(B,D),$\newline
(iv) $\mathcal{H}(\lambda A,\mu A)=|\lambda -\mu |\mathcal{H}(A,\theta )$
for $\lambda \mu \geq 0,$\newline
(v) $\mathcal{H}(A\ominus B,C)=\mathcal{H}(A,B+C)$\textit{,}\newline
(vi) $\mathcal{H}(A,B)=\mathcal{H}(A\ominus C,B\ominus C)$\textit{,}\newline
(vii) $\mathcal{H}(A\ominus B,C\ominus D)\leq \mathcal{H}(A,C)+\mathcal{H}%
(B,D)$\textit{,}\newline
(viii) $\mathcal{H}(A+B,C+D+E\ominus F)\leq \mathcal{H}(A,C+E)+\mathcal{H}%
(D,B+F)$\textit{,}\newline
(ix) $\mathcal{H}(A\ominus B,C\ominus D+E+F)\leq \mathcal{H}(A,C+E)+\mathcal{%
H}(D,B+F)$\textit{.}\newline

for all $A,B,C,D,E,F\in \mathcal{K}.$
\end{proposition}

\vskip0.3cm

\noindent Since $\mathcal{K}$\ is a normed quasilinear space, the continuity
and the limits of an interval-valued function $F:[a,b]\rightarrow \mathcal{K}
$ are understood in the sense of the norm $\left\Vert \cdot \right\Vert $.
We recall that if $F:[a,b]\rightarrow \mathcal{K}$ is an interval-valued
function such that $F(t)=[f^{-}(t),f^{+}(t)]$, then $\underset{t\rightarrow
t_{0}}{\lim }F(t)$ exists, if and only if $\underset{t\rightarrow t_{0}}{%
\lim }f^{-}(t)$ and $\underset{t\rightarrow t_{0}}{\lim }f^{+}(t)$ exist as
finite numbers. In this case, we have%
\begin{equation*}
\underset{t\rightarrow t_{0}}{\lim }F(t)=\left[ \underset{t\rightarrow t_{0}}%
{\lim }f^{-}(t),\underset{t\rightarrow t_{0}}{\lim }f^{+}(t)\right] .
\end{equation*}%
In particular, $F\ $is continuous if and only if $f^{-}$ and $f^{+}$ are
continuous. If\textit{\ }$F,G:[a,b]\rightarrow \mathcal{K}$\textit{\ }are
two interval-valued functions, then we define the interval-valued function $%
F\ominus G:[a,b]\rightarrow \mathcal{K}$ by $(F\ominus G)(t):=F(t)\ominus
G(t)$, for all $t\in \lbrack a,b]$. If there exist $\underset{t\rightarrow
t_{0}}{\lim }F(t)=A$\ and $\underset{t\rightarrow t_{0}}{\lim }G(t)=B$, then 
$\underset{t\rightarrow t_{0}}{\lim }(F\ominus G)(t)$\ exists, and%
\begin{equation*}
\underset{t\rightarrow t_{0}}{\lim }(F\ominus G)(t)=A\ominus B.
\end{equation*}%
In particular, if $F,G:[a,b]\rightarrow \mathcal{K}$\textit{\ }are
continuous, then the interval-function $F\ominus G\ $is a continuous
interval-valued function. Let $C([a,b],\mathcal{K})$ denote the set of
continuous interval-valued functions from $[a,b]$ into $\mathcal{K}$. Then $%
C([a,b],\mathcal{K})$ is a complete normed space with respect to the norm $%
\left\Vert F\right\Vert _{c}:=\underset{a\leq t\leq b}{\sup }\left\Vert
F(t)\right\Vert $. \vskip0.3cm

\begin{definition}
\label{def-1}\textbf{(}Markov\textbf{\ }\cite{mar}). Let $F:[a,b]\rightarrow 
\mathcal{K}$ be an interval-valued function and let $t_{0}\in \lbrack a,b]$.
We define $D_{H}F(t_{0})\in \mathcal{K}$ (provided it exists) as%
\begin{equation}
D_{H}F(t_{0})=\underset{h\rightarrow 0}{\lim }\frac{F(t_{0}+h)\ominus
F(t_{0})}{h}.  \label{gH-der}
\end{equation}%
We call $D_{H}F(t_{0})$ the \emph{generalized Hukuhara derivative} ($gH$-%
\emph{derivative} for short) of $F$ at $t_{0}$. Also, we define the \emph{%
left }$gH$-\emph{derivative }$D_{H}^{-}F(t_{0})\in \mathcal{K}$ (provided it
exists) as%
\begin{equation*}
D_{H}^{-}F(t_{0})=\underset{h\rightarrow 0^{+}}{\lim }\frac{F(t_{0})\ominus
F(t_{0}-h)}{h},.
\end{equation*}%
and the \emph{right }$gH$-\emph{derivative }$D_{H}^{+}F(t_{0})\in \mathcal{K}
$ (provided it exists) as%
\begin{equation*}
D_{H}^{+}F(t_{0})=\underset{h\rightarrow 0^{+}}{\lim }\frac{%
F(t_{0}+h)\ominus F(t_{0})}{h}.
\end{equation*}%
We say that $F$ is \emph{generalized Hukuhara differentiable} ($gH$-\emph{%
differentiable} for short) on $[a,b]$ if $D_{H}F(t_{0})\in \mathcal{K}$
exists at each point $t\in \lbrack a,b]$. At the end points of $[a,b]$ we
consider only the one sided $gH$-derivatives. The interval-valued function $%
D_{H}:[a,b]\rightarrow \mathcal{K}$ is then called the $gH$-\emph{derivative}
of $F$ on $[a,b]$. \vskip0.3cm
\end{definition}

\begin{proposition}
\label{p-2}\textbf{(}Markov\textbf{\ }\cite{mar}). \textit{Let }$%
F:[a,b]\rightarrow \mathcal{K}$\textit{\ be such that }$%
F(t)=[f^{-}(t),f^{+}(t)]$, $t\in \lbrack a,b]$. \textit{If the real-valued
functions} $f^{-}$ \textit{and} $f^{+}$\ \textit{are} \textit{differentiable
at }$t\in \lbrack a,b]$, \textit{then }$F$ \textit{is} $gH$-\textit{%
differentiable at }$t\in \lbrack a,b]$ \textit{and}%
\begin{equation}
D_{H}F(t_{0})=\left[ \min \left\{ \frac{d}{dt}f^{-}(t),\frac{d}{dt}%
f^{+}(t)\right\} ,\max \left\{ \frac{d}{dt}f^{-}(t),\frac{d}{dt}%
f^{+}(t)\right\} \right] .  \label{deriv_1}
\end{equation}
\end{proposition}

\noindent The converse of Proposition \ref{p-2} does not true, that is, the $%
gH$-differentiability of $F$ does not imply the differentiability of $f^{-}$
and $f^{+}$ (Markov \cite{mar}).\vskip0.3cm

\section{\protect\Large A new concept of differentiability for
interval-valued functions}

\noindent Let $F:[a,b]\rightarrow \mathcal{K}$ be a given function. We say
that $F$ is \emph{left differentiable} at $t_{0}\in (a,b]$ if there exists
an element $A\in \mathcal{K}$ such that%
\begin{equation}
\underset{h\rightarrow 0^{+}}{\lim }\frac{1}{h}\mathcal{H}%
(F(t_{0}),F(t_{0}-h)+hA)=0\text{ }  \label{der_11}
\end{equation}%
or 
\begin{equation}
\underset{h\rightarrow 0^{+}}{\lim }\frac{1}{h}\mathcal{H}%
(F(t_{0}-h),F(t_{0})-hA)=0.  \label{der_12}
\end{equation}%
The element $A\in \mathcal{K}$ is called a \emph{left derivative} of $F$ at $%
t_{0}$ and it will be denoted by $F_{-}^{\prime }(t_{0})$. $F$ is said to be 
\emph{left differentiable} on $(a,b]$, if $F$ is left differentiable at each 
$t_{0}\in (a,b]$.\vskip0.3cm

\begin{remark}
\label{rem-1} Let $F:[a,b]\rightarrow \mathcal{K}$ be a given function. If
it exists, the left derivative of $F$ at a point $t_{0}\in (a,b]$\ is
unique. Indeed, suppose that $A(t_{0}),B(t_{0})$ are left derivatives of $F$
at $t_{0}\in (a,b]$. Then from the properties of the metric $\mathcal{H}$
and (\ref{der_11}) or (\ref{der_12}) it follows that%
\begin{eqnarray*}
\mathcal{H}(A(t_{0}),B(t_{0})) &=&\frac{1}{h}\mathcal{H}(hA(t_{0}),hA(t_{0}))
\\
&=&\frac{1}{h}\mathcal{H}(F(t_{0}-h)+hA(t_{0}),F(t_{0}-h)+hB(t_{0})) \\
&\leq &\frac{1}{h}\mathcal{H}(F(t_{0}-h)+hA(t_{0}),F(t_{0})) \\
&&+\frac{1}{h}\mathcal{H}(F(t_{0}-h)+hB(t_{0}),F(t_{0})) \\
&\rightarrow &0\text{ as }h\rightarrow 0^{+}
\end{eqnarray*}%
or%
\begin{eqnarray*}
\mathcal{H}(A(t_{0}),B(t_{0})) &=&\frac{1}{h}\mathcal{H}%
(h(-A(t_{0})),h(-A(t_{0}))) \\
&=&\frac{1}{h}\mathcal{H}(F(t_{0}-h)-hA(t_{0}),F(t_{0}-h)-hB(t_{0})) \\
&\leq &\frac{1}{h}\mathcal{H}(F(t_{0}-h)-hA(t_{0}),F(t_{0})) \\
&&+\frac{1}{h}\mathcal{H}(F(t_{0}-h)-hB(t_{0}),F(t_{0})) \\
&\rightarrow &0\text{ as }h\rightarrow 0^{+}\text{,}
\end{eqnarray*}%
respectively. Therefore, $\mathcal{H}(A(t_{0}),B(t_{0}))=0$ and so $%
A(t_{0})=B(t_{0})$. Also, we remark that the conditions (\ref{der_11}) and (%
\ref{der_12}) may not be equivalent as we see in the following
example.\smallskip 
\end{remark}

\begin{example}
\label{ex-3} Consider the function $F:\mathbb{R}\rightarrow \mathcal{K}$
given by $F(t)=[0,|t|]$. For $A=[a^{-},a^{+}]\in \mathcal{K}$, we have that 
\begin{equation*}
\begin{tabular}{l}
$\underset{h\rightarrow 0^{+}}{\lim }\frac{1}{h}\mathcal{H}(F(0-h),F(0)-hA)=%
\underset{h\rightarrow 0^{+}}{\lim }\frac{1}{h}\mathcal{H}%
([0,h],[-ha^{+},-ha^{-}]$ \\ 
\\ 
$=\underset{h\rightarrow 0^{+}}{\lim }\frac{1}{h}\max \left\{
|ha^{+}|,|h+ha^{-}|\right\} =\max \left\{ |a^{+}|,|1+a^{-}|\right\} =0$%
\end{tabular}%
\end{equation*}%
only if $a^{-}=-1<a^{+}=0$. It follows that $\ $(\ref{der_12}) holds.
Therefore, $F$ is left differentiable at $t_{0}=0$ and $F_{-}^{\prime
}(0)=[-1,0]$. Since $\underset{h\rightarrow 0^{+}}{\lim }\frac{1}{h}\mathcal{%
H}(F(0),F(0-h)+hA)=0$ only if $a^{-}=0>a^{+}=-1$, then there exists a
contradiction with the assumption that $A=[a^{-},a^{+}]\in \mathcal{K}$, and
so (\ref{der_11}) does not hold.\vskip0.3cm
\end{example}

\noindent We say that $F$ is \emph{right differentiable} at $t_{0}\in
\lbrack a,b)$ if there exists an element $A\in \mathcal{K}$ such that%
\begin{equation}
\underset{h\rightarrow 0^{+}}{\lim }\frac{1}{h}\mathcal{H}%
(F(t_{0}+h),F(t_{0})+hA)=0  \label{der_21}
\end{equation}%
or 
\begin{equation}
\underset{h\rightarrow 0^{+}}{\lim }\frac{1}{h}\mathcal{H}%
(F(t_{0}),F(t_{0}+h)-hA)=0.  \label{der_22}
\end{equation}%
The element $A\in \mathcal{K}$ is called a \emph{right derivative} of $F$ at 
$t_{0}$ and it will be denoted by $F_{+}^{\prime }(t_{0})$. $F$ is said to
be \emph{right differentiable} on $[a,b)$, if $F$ is right differentiable at
each $t_{0}\in \lbrack a,b)$. \vskip0.3cm

\begin{remark}
\label{rem-2} Using the same reasoning as in Remark \ref{rem-1}, we can show
that if it exists, the right derivative of $F$ at a point $t_{0}\in \lbrack
a,b)$\ is unique. Also, we remark that the conditions (\ref{der_21}) and (%
\ref{der_22}) may not be equivalent as we see in the following example.\vskip%
0.3cm
\end{remark}

\begin{example}
\label{ex-4} Consider the function $F:\mathbb{R}\rightarrow \mathcal{K}$
given by $F(t)=[0,|t|]$. For $A=[a^{-},a^{+}]\in \mathcal{K}$, we have that%
\begin{equation*}
\begin{tabular}{l}
$\underset{h\rightarrow 0^{+}}{\lim }\frac{1}{h}\mathcal{H}(F(0+h),F(0)+hA)=%
\underset{h\rightarrow 0^{+}}{\lim }\frac{1}{h}\mathcal{H}%
([0,h],[ha^{-},ha^{+}]$ \\ 
\\ 
$=\underset{h\rightarrow 0^{+}}{\lim }\frac{1}{h}\max \left\{
|ha^{+}|,|h-ha^{-}|\right\} =\max \left\{ |a^{+}|,|1-a^{-}|\right\} =0$%
\end{tabular}%
\end{equation*}%
only if $a^{-}=0<a^{+}=1$. It follows that $\ $(\ref{der_21}) holds.
Therefore, $F$ is right differentiable at $t_{0}=0$ and $F_{+}^{\prime }(0)=%
\left[ 0,1\right] $. Since $\underset{h\rightarrow 0^{+}}{\lim }\frac{1}{h}%
\mathcal{H}(F(0),F(0+h)-hA)=0$ only if $a^{-}=1>a^{+}=0$, then there exists
a contradiction with the assumption that $A=[a^{-},a^{+}]\in \mathcal{K}$,
and so (\ref{der_22}) does not hold.\vskip0.3cm
\end{example}

\noindent We say that $F$ is \emph{differentiable} at $t_{0}\in \lbrack a,b]$
if $F$ is left and right differentiable at $t_{0}$, and $F_{-}^{\prime
}(t_{0})=F_{+}^{\prime }(t_{0})$. The element $F_{-}^{\prime }(t_{0})$ or $%
F_{+}^{\prime }(t_{0})$ will be denoted by $F^{\prime }(t_{0})$ and it is
called a \emph{derivative} of $F$ at $t_{0}$. $F$ is said to be\emph{\
differentiable} on $[a,b]$, if $F$ is differentiable at each $t_{0}\in
\lbrack a,b]$. At the end points of $[a,b]$, we consider only the one-side
derivatives.\vskip0.3cm

\begin{remark}
\label{rem-3} From Remarks \ref{rem-1} and \ref{rem-2}, it is clear if it
exists, the derivative of $F$ at a point $t_{0}\in \lbrack a,b]$\ is unique.%
\vskip0.3cm
\end{remark}

\begin{example}
\label{ex-5} Let $F:\mathbb{R}\rightarrow \mathcal{K}$ be the function given
by $F(t)=[0,|t|]$. From Examples \ref{ex-4} and \ref{ex-5} we have that $%
F_{-}^{\prime }(0)\neq F_{+}^{\prime }(0)$, and so $F$ is not differentiable
at $t_{0}=0$. \vskip0.3cm
\end{example}

\begin{theorem}
\label{th-2} \textit{If }$F:[a,b]\rightarrow \mathcal{K}$ \textit{is left
(right) differentiable at }$t_{0}\in (a,b]$ \textit{(}$t_{0}\in \lbrack a,b)$%
\textit{), then }$F$ \textit{is left (right) continuous at }$t_{0}$\textit{.
In particular, if }$F$ \textit{is differentiable at }$t_{0}$\textit{, then }$%
F$ \textit{is continuous at }$t_{0}$\textit{.}\vskip0.3cm
\end{theorem}

\begin{proof}
Suppose that $F$ is left differentiable at $t_{0}$ and $F_{-}^{\prime
}(t_{0})=A$ and let $\varepsilon >0$. Then from (\ref{der_11}) or (\ref%
{der_12}) it follows that there exists a $\delta >0$ such that for all $h\in
(0,\delta )$ we have%
\begin{eqnarray*}
\mathcal{H}(F(t_{0}-h),F(t_{0})) &\leq &\mathcal{H}(F(t_{0}-h),F(t_{0})+hA)+%
\mathcal{H}(F(t_{0})+hA,F(t_{0})) \\
&=&\mathcal{H}(F(t_{0}-h),F(t_{0})+hA)+\mathcal{H}(hA,\theta ) \\
&\leq &\varepsilon h+h\mathcal{H}(A,\theta )
\end{eqnarray*}%
or%
\begin{eqnarray*}
\mathcal{H}(F(t_{0}-h),F(t_{0})) &\leq &\mathcal{H}(F(t_{0}-h),F(t_{0})-hA)+%
\mathcal{H}(F(t_{0})-hA,F(t_{0})) \\
&=&\mathcal{H}(F(t_{0}-h),F(t_{0})-hA)+\mathcal{H}(h(-A),\theta ) \\
&\leq &\varepsilon h+h\mathcal{H}(-A,\theta ),
\end{eqnarray*}%
respectively. Therefore, $\underset{h\rightarrow 0^{+}}{\lim }\mathcal{H}%
(F(t_{0}-h),F(t_{0}))=0$, and so $F$ is left continuous at $t_{0}$. The
proof is similar when it is assumed that $F$ is right differentiable at $%
t_{0}$. 
\end{proof}

\vskip0.3cm

\begin{remark}
\label{rem-4} From the definition it follows that a function $%
F:[a,b]\rightarrow \mathcal{K}$ is differentiable at $t_{0}\in \lbrack a,b]$
if there exists an $A\in \mathcal{K}$ such that\ one of the following
conditions is true%
\begin{equation}
\underset{h\rightarrow 0^{+}}{\lim }\frac{1}{h}\mathcal{H}%
(F(t_{0}+h),F(t_{0})+hA)=\underset{h\rightarrow 0^{+}}{\lim }\frac{1}{h}%
\mathcal{H}(F(t_{0}),F(t_{0}-h)+hA)=0,  \label{der_1}
\end{equation}%
\begin{equation}
\underset{h\rightarrow 0^{+}}{\lim }\frac{1}{h}\mathcal{H}%
(F(t_{0}),F(t_{0}+h)-hA)=\underset{h\rightarrow 0^{+}}{\lim }\frac{1}{h}%
\mathcal{H}(F(t_{0}-h),F(t_{0})-hA)=0,  \label{der_2}
\end{equation}%
\begin{equation}
\underset{h\rightarrow 0^{+}}{\lim }\frac{1}{h}\mathcal{H}%
(F(t_{0}),F(t_{0}+h)-hA)=\underset{h\rightarrow 0^{+}}{\lim }\frac{1}{h}%
\mathcal{H}(F(t_{0}),F(t_{0}-h)+hA)=0,  \label{der_3}
\end{equation}%
\begin{equation}
\underset{h\rightarrow 0^{+}}{\lim }\frac{1}{h}\mathcal{H}%
(F(t_{0}+h),F(t_{0})+hA)=\underset{h\rightarrow 0^{+}}{\lim }\frac{1}{h}%
\mathcal{H}(F(t_{0}-h),F(t_{0})-hA)=0.  \label{der_4}
\end{equation}%
The element $A\in \mathcal{K}$ is the derivative of $F$ at $t_{0}$; that is, 
$A=F^{\prime }(t_{0})$. 
\end{remark}

\vskip0.3cm

\begin{example}
\label{ex-6} Consider the function $F:\mathbb{R}\rightarrow \mathcal{K}$
given by $F(t)=[-t^{2},t^{2}]$. For $A=[a^{-},a^{+}]\in \mathcal{K}$, we
have that 
\begin{equation*}
\begin{tabular}{l}
$\underset{h\rightarrow 0^{+}}{\lim }\frac{1}{h}\mathcal{H}(F(t),F(t-h)+hA)=$
\\ 
\\ 
$=\underset{h\rightarrow 0^{+}}{\lim }\frac{1}{h}\max
\{|2th-h^{2}+ha^{-}|,|-2th+h^{2}+ha^{+}|\}$ \\ 
\\ 
$=\max \left\{ |2t+a^{-}|,|-2t+a^{+}|\right\} =0$%
\end{tabular}%
\end{equation*}%
only if $a^{-}=-2t\leq a^{+}=2t$ and $t\geq 0$. Also, we have that%
\begin{equation*}
\begin{tabular}{l}
$\underset{h\rightarrow 0^{+}}{\lim }\frac{1}{h}\mathcal{H}(F(t-h),F(t)-hA)=$
\\ 
\\ 
$=\underset{h\rightarrow 0^{+}}{\lim }\frac{1}{h}\max
\{|2th-h^{2}+ha^{+}|,|-2th+h^{2}+ha^{-}|\}$ \\ 
\\ 
$=\max \{|2t+a^{+}|,|-2t+a^{-}|\}=0$%
\end{tabular}%
\end{equation*}%
only if $a^{-}=2t<a^{+}=-2t$ and $t<0$. It follows that $F$ is left
differentiable at $t\in \mathbb{R}$ and%
\begin{equation*}
F_{-}^{\prime }(t)=\left\{ 
\begin{tabular}{ll}
$\lbrack -2t,2t]$ & if $t\geq 0$ \\ 
$\lbrack 2t,-2t]$ & if $t<0.$%
\end{tabular}%
\right. 
\end{equation*}%
Further, we have that%
\begin{equation*}
\begin{tabular}{l}
$\underset{h\rightarrow 0^{+}}{\lim }\frac{1}{h}\mathcal{H}(F(t+h),F(t)+hA)=$
\\ 
\\ 
$=\underset{h\rightarrow 0^{+}}{\lim }\frac{1}{h}\max
\{|-2th-h^{2}-ha^{-}|,|2th+h^{2}-ha^{+}|\}$ \\ 
\\ 
$=\max \left\{ |-2t-a^{-}|,|2t-a^{+}|\right\} =0$%
\end{tabular}%
\end{equation*}%
only if $a^{-}=-2t\leq a^{+}=2t$ and $t\geq 0$. Also, we have that%
\begin{equation*}
\begin{tabular}{l}
$\underset{h\rightarrow 0^{+}}{\lim }\frac{1}{h}\mathcal{H}_{\mathcal{I}%
}(F(t),F(t+h)-hA)=$ \\ 
\\ 
$=\underset{h\rightarrow 0^{+}}{\lim }\frac{1}{h}\max
\{|-2th-h^{2}-ha^{+}|,|2th+h^{2}-ha^{-}|\}$ \\ 
\\ 
$=\max \{|-2t-a^{+}|,|2t-a^{-}|\}=0$%
\end{tabular}%
\end{equation*}%
only if $a^{-}=2t<a^{+}=-2t$ and $t<0$. Therefore, $F$ is right
differentiable at each $t\in \mathbb{R}$ and%
\begin{equation*}
F_{+}^{\prime }(t)=\left\{ 
\begin{tabular}{ll}
$\lbrack -2t,2t]$ & if $t\geq 0$ \\ 
$\lbrack 2t,-2t]$ & if $t<0.$%
\end{tabular}%
\right. 
\end{equation*}%
Since $F_{+}^{\prime }(t)=F_{-}^{\prime }(t)$ for all $t\in \mathbb{R}$, it
follows that $F$ is differentiable at each $t\in \mathbb{R}$. We remark that%
\begin{equation*}
\underset{h\rightarrow 0^{+}}{\lim }\frac{1}{h}\mathcal{H}(F(t+h),F(t)+hA)=%
\underset{h\rightarrow 0^{+}}{\lim }\frac{1}{h}\mathcal{H}(F(t),F(t-h)+hA)=0
\end{equation*}%
for $A=[-2t,2t]$ and $t>0$; that is, (\ref{der_1}) holds for each $t>0$, but
(\ref{der_2})-(\ref{der_4}) do not hold for $t>0$. Also,%
\begin{equation*}
\underset{h\rightarrow 0^{+}}{\lim }\frac{1}{h}\mathcal{H}_{\mathcal{I}%
}(F(t),F(t+h)-hA)=\underset{h\rightarrow 0^{+}}{\lim }\frac{1}{h}\mathcal{H}%
_{\mathcal{I}}(F(t-h),F(t)-hA)=0
\end{equation*}%
for $A=[2t,-2t]$ and $t<0$; that is, (\ref{der_4}) holds for each $t>0$, but
(\ref{der_1})-(\ref{der_3}) do not hold for $t>0$. If $t=0$, then it is easy
to check that $F$ is differentiable at $t=0$, $F^{\prime }(0)=\{0\}$, and (%
\ref{der_1})-(\ref{der_4}) are equivalent for $A=\{0\}$. \vskip0.3cm
\end{example}

\begin{remark}
\label{re-5} In \cite{lak1}, a function $F:[a,b]\rightarrow \mathcal{K}$ is
called differentiable at $t_{0}\in \lbrack a,b]$ if there exists an element $%
A\in \mathcal{F}$ such that%
\begin{equation*}
\underset{h\rightarrow 0^{+}}{\lim }\frac{1}{h}\mathcal{H}%
(F(t_{0}+h),F(t_{0})+hA)=0.
\end{equation*}%
In this case, the element $A\in \mathcal{K}$ is called the derivative of $F$
at $t_{0}$. In \cite{stef1} it is shown that the function $F:\mathbb{R}%
\rightarrow \mathcal{K}$ given by $F(t)=[e^{-t},2e^{-t}]$ is not
differentiable in this sense since, for a $t_{0}\in \mathbb{R}$ and $%
A=[a^{-},a^{+}]\in \mathcal{K}$, we have that $\underset{h\rightarrow 0^{+}}{%
\lim }\frac{1}{h}\mathcal{H}_{\mathcal{I}}(F(t_{0}+h),F(t_{0})+hA)=0$ only
if $a^{-}=-e^{-t_{0}}>a^{+}=-2e^{-t_{0}}$ which is a contradiction with the
assumption that $A=[a^{-},a^{+}]\in \mathcal{K}$. However, for $t_{0}\in 
\mathbb{R}$ and $A=[-2e^{-t_{0}},-e^{-t_{0}}]\in \mathcal{K}$ we have that 
\begin{equation*}
\begin{tabular}{l}
$\underset{h\rightarrow 0^{+}}{\lim }\frac{1}{h}\mathcal{H}%
(F(t_{0}),F(t_{0}+h)-hA)=$ \\ 
\\ 
$=\underset{h\rightarrow 0^{+}}{\lim }\frac{1}{h}\mathcal{H}%
([e^{-t_{0}},2e^{-t_{0}}],[e^{-t_{0}-h},2e^{-t_{0}-h}]-h[-2e^{-t_{0}},-e^{-t_{0}}])
$ \\ 
\\ 
$=\underset{h\rightarrow 0^{+}}{\lim }\frac{1}{h}\mathcal{H}%
([e^{-t_{0}},2e^{-t_{0}}],[e^{-t_{0}-h}+he^{-t_{0}},2e^{-t_{0}-h}+2he^{-t_{0}}])
$ \\ 
\\ 
$=\underset{h\rightarrow 0^{+}}{\lim }\frac{1}{h}\max \left\{
|e^{-t_{0}-h}-e^{-t_{0}}+he^{-t_{0}}|,2|e^{-t_{0}-h}-e^{-t_{0}}+he^{-t_{0}}|%
\right\} $ \\ 
\\ 
$=\underset{h\rightarrow 0^{+}}{\lim }2|e^{-t_{0}}\frac{e^{-h}-1}{h}%
+e^{-t_{0}}|=0$%
\end{tabular}%
\end{equation*}%
and 
\begin{equation*}
\begin{tabular}{l}
$\underset{h\rightarrow 0^{+}}{\lim }\frac{1}{h}\mathcal{H}%
(F(t_{0}-h),F(t_{0})-hA)=$ \\ 
\\ 
$=\underset{h\rightarrow 0^{+}}{\lim }\frac{1}{h}\mathcal{H}%
([e^{-t_{0}+h},2e^{-t_{0}+h}],[e^{-t_{0}},2e^{-t_{0}}]-h[-2e^{-t_{0}},-e^{-t_{0}}])
$ \\ 
\\ 
$=\underset{h\rightarrow 0^{+}}{\lim }\frac{1}{h}\mathcal{H}%
([e^{-t_{0}+h},2e^{-t_{0}+h}],[e^{-t_{0}}+he^{-t_{0}},2e^{-t_{0}}+2he^{-t_{0}}])
$ \\ 
\\ 
$=\underset{h\rightarrow 0^{+}}{\lim }\frac{1}{h}\max \left\{
|e^{-t_{0}+h}-e^{-t_{0}}-he^{-t_{0}}|,2|e^{-t_{0}+h}-e^{-t_{0}}-he^{-t_{0}}|%
\right\} $ \\ 
\\ 
$=\underset{h\rightarrow 0^{+}}{\lim }2|e^{-t_{0}}\frac{e^{h}-1}{h}%
-e^{-t_{0}}|=0.$%
\end{tabular}%
\end{equation*}%
It follows that%
\begin{equation*}
\underset{h\rightarrow 0^{+}}{\lim }\frac{1}{h}\mathcal{H}%
(F(t_{0}),F(t_{0}+h)-hA)=\underset{h\rightarrow 0^{+}}{\lim }\frac{1}{h}%
\mathcal{H}(F(t_{0}-h),F(t_{0})-hA)=0.
\end{equation*}%
Therefore, (\ref{der_2}) holds and so $F$ is differentiable at $t_{0}$ with $%
F^{\prime }(t_{0})=\left[ -2e^{-t_{0}},-e^{-t_{0}}\right] $.\vskip0.3cm
\end{remark}

\begin{theorem}
\label{th-3} \textit{Let }$F:[a,b]\rightarrow \mathcal{K}$ \textit{be a
given function. If there exists an }$A\in \mathcal{K}$\textit{\ such that\ (%
\ref{der_11}) and (\ref{der_12}) or (\ref{der_21}) and (\ref{der_22}) occur
simultaneously, then} $A$ is a singleton.\vskip0.3cm
\end{theorem}

\begin{proof}
Suppose that (\ref{der_21}) and (\ref{der_22}) simultaneously hold. Since%
\begin{equation*}
\begin{tabular}{l}
$\mathcal{H}(A-A,\theta )=\frac{1}{h}\mathcal{H}(hA-hA,\theta )=\frac{1}{h}%
\mathcal{H}(F(t_{0})+hA-hA,F(t_{0}))$ \\ 
\\ 
$\leq \frac{1}{h}\mathcal{H}(F(t_{0})+hA-hA,F(t_{0}+h)-hA)+\frac{1}{h}%
\mathcal{H}(F(t_{0}+h)-hA,F(t_{0}))$ \\ 
\\ 
$=\frac{1}{h}\mathcal{H}(F(t_{0})+hA,F(t_{0}+h))+\frac{1}{h}\mathcal{H}%
(F(t_{0}),F(t_{0}+h)-hA)$ \\ 
\\ 
$\rightarrow 0$ as $h\rightarrow 0^{+},$%
\end{tabular}%
\end{equation*}%
it follows that $\mathcal{H}(A-A,\theta )=0$. Therefore, $A-A=\theta $ and
so $A\ $is a singleton. A similar proof establishes the result if (\ref%
{der_11}) and (\ref{der_12}) simultaneously hold\textit{.}
\end{proof}

\vskip0.3cm

\begin{corollary}
\label{cor-2} \textit{If for a given function }$F:[a,b]\rightarrow \mathcal{K%
}$ \textit{and }$t_{0}\in \lbrack a,b]$\textit{\ there exists an }$A\in 
\mathcal{K}$\textit{\ such that at least two from the conditions (\ref{der_1}%
)-(\ref{der_4}) occur simultaneously, then }$A\ $is a singleton.\vskip0.3cm
\end{corollary}

\begin{proposition}
\label{p-9} \textit{If }$F:[a,b]\rightarrow \mathcal{K}$ \textit{is a given
function and }$t_{0}\in (a,b)$\textit{. Then the following statements are
true.}

(a)\textit{\ If there exists }$A\in \mathcal{K}$\textit{\ such that\ (\ref%
{der_1}) holds, then}%
\begin{equation}
\underset{h\rightarrow 0^{+}}{\lim }\frac{1}{h}\mathcal{H}%
(F(t_{0}+h),F(t_{0}-h)+2hA)=0.  \label{s1}
\end{equation}%
(b)\textit{\ If there exists }$A\in \mathcal{K}$\textit{\ such that\ (\ref%
{der_2}) holds, then}%
\begin{equation}
\underset{h\rightarrow 0^{+}}{\lim }\frac{1}{h}\mathcal{H}%
(F(t_{0}-h),F(t_{0}+h)-2hA)=0.  \label{s2}
\end{equation}%
(c)\textit{\ If there exists }$A\in \mathcal{K}$\textit{\ such that\ (\ref%
{der_3}) or (\ref{der_4}) holds, then}%
\begin{equation}
\underset{h\rightarrow 0^{+}}{\lim }\frac{1}{h}\mathcal{H}%
(F(t_{0}+h)-hA,F(t_{0}-h)+hA)=0.  \label{s3}
\end{equation}
\end{proposition}

\begin{proof}
Suppose there exists $A\in \mathcal{K}$\ such that\ (\ref{der_1}) holds.
Then we have%
\begin{eqnarray*}
&&\frac{1}{h}\mathcal{H}(F(t_{0}+h),F(t_{0}-h)+2hA) \\
&=&\frac{1}{h}\mathcal{H}(F(t_{0})+F(t_{0}+h),F(t_{0})+F(t_{0}-h)+2hA) \\
&\leq &\frac{1}{h}\mathcal{H}(F(t_{0}),F(t_{0}-h)+hA)+\frac{1}{h}\mathcal{H}%
(F(t_{0}+h),F(t_{0})+hA)\rightarrow 0
\end{eqnarray*}%
as $h\rightarrow 0^{+}$, and so (\ref{s1}) is true. With a similar reasoning
we can prove statement (b). Now, suppose that there exists $A\in \mathcal{K}$%
\ such that\ (\ref{der_3}) holds. Then we have%
\begin{eqnarray*}
&&\frac{1}{h}\mathcal{H}(F(t_{0}+h)-hA,F(t_{0}-h)+hA) \\
&=&\frac{1}{h}\mathcal{H}(F(t_{0})+F(t_{0}+h)-hA,F(t_{0})+F(t_{0}-h)+hA) \\
&\leq &\frac{1}{h}\mathcal{H}(F(t_{0}),F(t_{0}-h)+hA)+\frac{1}{h}\mathcal{H}%
(F(t_{0}),F(t_{0}+h)-hA)\rightarrow 0
\end{eqnarray*}%
as $h\rightarrow 0^{+}$, and so (\ref{s3}) is true. If there exists an $A\in 
\mathcal{K}$\ such that\ (\textit{\ref{der_4}}) holds, then we have%
\begin{eqnarray*}
&&\frac{1}{h}\mathcal{H}(F(t_{0}+h)-hA,F(t_{0}-h)+hA) \\
&=&\frac{1}{h}\mathcal{H}(F(t_{0})+F(t_{0}+h)-hA,F(t_{0})+F(t_{0}-h)+hA) \\
&\leq &\frac{1}{h}\mathcal{H}(F(t_{0})-hA,F(t_{0}-h))+\frac{1}{h}\mathcal{H}%
(F(t_{0}+h),F(t_{0})+hA)\rightarrow 0\text{ }
\end{eqnarray*}%
as $h\rightarrow 0^{+}$, and so (\ref{s3}) is again true. 
\end{proof}

\vskip0.3cm

\begin{theorem}
\label{th-4}\textit{\ Let }$F:[a,b]\rightarrow \mathcal{K}$ \textit{be a
given function. If there exists an }$A\in \mathcal{K}$\textit{\ such that\ (%
\ref{der_3}) or (\ref{der_4}) holds, then} $A\ $is a singleton.\vskip0.3cm
\end{theorem}

\begin{proof}
Suppose that (\ref{der_3}) holds. Then (\ref{s3}) is also true and it
follows that 
\begin{equation*}
\begin{tabular}{l}
$\left\vert \frac{1}{h}d(F(t_{0}),F(t_{0}+h)-hA)-\frac{1}{h}%
d(F(t_{0}),F(t_{0}-h)+hA)\right\vert $ \\ 
\\ 
$\leq \frac{1}{h}d(F(t_{0}+h)-hA,F(t_{0}-h)+hA)\rightarrow 0$ as $%
h\rightarrow 0^{+};$%
\end{tabular}%
\end{equation*}%
that is, 
\begin{equation}
\underset{h\rightarrow 0^{+}}{\lim }\frac{1}{h}d(F(t_{0}),F(t_{0}+h)-hA)=%
\underset{h\rightarrow 0^{+}}{\lim }\frac{1}{h}d(F(t_{0}),F(t_{0}-h)+hA).
\label{d-1}
\end{equation}%
On the other hand, if we take in (\ref{der_11}) $t=t+h$ we obtain $\underset{%
h\rightarrow 0^{+}}{\lim }\frac{1}{h}d(F(t_{0}),F(t_{0}+h)-hA)=0$. Then from
(\ref{d-1}) it follows that (\ref{der_4}) also holds. Therefore (\ref{der_3}%
) and (\ref{der_4}) occur simultaneously and thus by Corollary \ref{cor-2}
we infer that $A\ $is a singleton. A similar proof works if (\ref{der_4})
holds.
\end{proof}

\vskip0.3cm

\begin{remark}
\label{rem-6} From Theorem \ref{th-3} and Theorem \ref{th-4} it follows that
a function $F:[a,b]\rightarrow \mathcal{K}$ can be differentiable in the
sense of (\ref{der_1}) or in the sense of (\ref{der_2}). We will say that $F$
is $\mathcal{H}^{1}$-\emph{differentiable} if it is differentiable in the
sense of (\ref{der_1}). Also, we will say that $F$ is $\mathcal{H}^{2}$-%
\emph{differentiable} if it is differentiable in the sense of (\ref{der_2}%
).\smallskip 
\end{remark}

\vskip0.3cm

\begin{theorem}
\label{th-7}\textit{\ Let }$F,G:[a,b]\rightarrow \mathcal{K}$ \textit{be two
given function. }

\textit{(a) If }$F$\textit{\ is differentiable and }$\lambda \in \mathbb{R}$%
\textit{, then the function }$\lambda F$ \textit{is differentiable and }$%
(\lambda F)^{\prime }=\lambda F^{\prime }$\textit{. }

\textit{(b) If }$F,G\in \mathcal{H}^{i}$\textit{\ (}$i=1,2$) \textit{and }$%
F\ominus G$, $F^{\prime }\ominus G^{\prime }$\textit{\ exist, then }$F+G$%
\textit{, }$F\ominus G\in \mathcal{H}^{i}\mathit{\ }(i=1,2)$ \textit{and}%
\begin{equation*}
\begin{tabular}{l}
$(F+G)^{\prime }=F^{\prime }+G^{\prime },$ \\ 
\\ 
$(F\ominus G)^{\prime }=F^{\prime }\ominus G^{\prime }.$%
\end{tabular}%
\end{equation*}

\textit{(c) If }$F\in \mathcal{H}^{i}$\textit{, }$G\in \mathcal{H}^{j}$%
\textit{\ (}$i,j=1,2$\textit{) for }$i\neq j$\textit{\ and }$F\ominus G$, $%
F^{\prime }\ominus (-G^{\prime })$\textit{\ exist, then }$F+G$, $F\ominus
G\in \mathcal{H}^{i}$\textit{\ and}%
\begin{equation*}
\begin{tabular}{l}
$(F+G)^{\prime }=F^{\prime }\ominus (-G^{\prime }),$ \\ 
\\ 
$(F\ominus G)^{\prime }=F^{\prime }+(-G^{\prime }).$%
\end{tabular}%
\end{equation*}
\end{theorem}

\begin{proof}
(a) is obvious. (b) Suppose that $F,G\in \mathcal{H}^{1}$. Using Proposition %
\ref{prop-3} and Propostion \ref{p-1} we have%
\begin{eqnarray*}
&&\underset{h\rightarrow 0^{+}}{\lim }\frac{1}{h}\mathcal{H}((F\ominus
G)(t+h),(F\ominus G)(t)+h(F^{\prime }\ominus G^{\prime })(t)) \\
&=&\underset{h\rightarrow 0^{+}}{\lim }\frac{1}{h}\mathcal{H}(F(t+h)\ominus
G(t+h),(F(t)+hF^{\prime }(t))\ominus (G(t)+hG^{\prime }(t))) \\
&\leq &\underset{h\rightarrow 0^{+}}{\lim }\frac{1}{h}\mathcal{H}%
(F(t+h),F(t)+hF^{\prime }(t)) \\
&&+\underset{h\rightarrow 0^{+}}{\lim }\frac{1}{h}\mathcal{H}%
(G(t+h),G(t)+hG^{\prime }(t))=0
\end{eqnarray*}%
and%
\begin{eqnarray*}
&&\underset{h\rightarrow 0^{+}}{\lim }\frac{1}{h}\mathcal{H}((F\ominus
G)(t),(F+G)(t-h)\ominus h(F^{\prime }+G^{\prime })(t)) \\
&=&\underset{h\rightarrow 0^{+}}{\lim }\frac{1}{h}\mathcal{H}%
(F(t)+G(t),(F(t-h)+hF^{\prime }(t))\ominus (G(t-h)+hG^{\prime }(t))) \\
&\leq &\underset{h\rightarrow 0^{+}}{\lim }\frac{1}{h}\mathcal{H}%
(F(t),F(t-h)+hF^{\prime }(t)) \\
&&+\underset{h\rightarrow 0^{+}}{\lim }\frac{1}{h}\mathcal{H}%
(G(t),G(t-h)+hG^{\prime }(t))=0.
\end{eqnarray*}%
It follows that $F\ominus G\in \mathcal{H}^{1}$ and $(F\ominus G)^{\prime
}=F^{\prime }\ominus G^{\prime }$. Also, it is easy to check that $F+G\in 
\mathcal{H}^{1}$ and $(F+G)^{\prime }=F^{\prime }+G^{\prime }$. A similar
proof establishes the result if $F,G\in \mathcal{H}^{2}$. (c) Suppose that $%
F\in \mathcal{H}^{1}$ and $G\in \mathcal{H}^{2}$. Then using Propostion \ref%
{p-1} we have 
\begin{eqnarray*}
&&\underset{h\rightarrow 0^{+}}{\lim }\frac{1}{h}\mathcal{H}%
((F+G)(t+h),(F+G)(t)+h(F^{\prime }\ominus (-G^{\prime }))(t)) \\
&=&\underset{h\rightarrow 0^{+}}{\lim }\frac{1}{h}\mathcal{H}%
(F(t+h)+G(t+h),F(t)+G(t)+hF^{\prime }(t))\ominus (-hG^{\prime }(t))) \\
&\leq &\underset{h\rightarrow 0^{+}}{\lim }\frac{1}{h}\mathcal{H}%
(F(t+h),F(t)+hF^{\prime }(t)) \\
&&+\underset{h\rightarrow 0^{+}}{\lim }\frac{1}{h}\mathcal{H}%
(G(t),G(t+h)-hG^{\prime }(t))=0
\end{eqnarray*}%
and%
\begin{eqnarray*}
&&\underset{h\rightarrow 0^{+}}{\lim }\frac{1}{h}\mathcal{H}%
((F+G)(t),(F+G)(t-h)+h(F^{\prime }\ominus (-G^{\prime }))(t)) \\
&=&\underset{h\rightarrow 0^{+}}{\lim }\frac{1}{h}\mathcal{H}%
(F(t)+G(t),(F(t-h)+G(t-h)+hF^{\prime }(t))\ominus (-hG^{\prime }(t))) \\
&\leq &\underset{h\rightarrow 0^{+}}{\lim }\frac{1}{h}\mathcal{H}%
(F(t),F(t-h)+hF^{\prime }(t)) \\
&&+\underset{h\rightarrow 0^{+}}{\lim }\frac{1}{h}\mathcal{H}%
(G(t),G(t-h)-hG^{\prime }(t))=0.
\end{eqnarray*}%
It follows that $F+G\in \mathcal{H}^{1}$ and $(F+G)^{\prime }=F^{\prime
}\ominus (-G^{\prime })$. Also we have%
\begin{eqnarray*}
&&\underset{h\rightarrow 0^{+}}{\lim }\frac{1}{h}\mathcal{H}((F\ominus
G)(t+h),(F\ominus G)(t)+h(F^{\prime }+(-G^{\prime }))(t)) \\
&=&\underset{h\rightarrow 0^{+}}{\lim }\frac{1}{h}\mathcal{H}(F(t+h)\ominus
G(t+h),F(t)\ominus G(t)+hF^{\prime }(t)-hG^{\prime }(t)) \\
&\leq &\underset{h\rightarrow 0^{+}}{\lim }\frac{1}{h}\mathcal{H}%
(F(t+h),F(t)+hF^{\prime }(t)) \\
&&+\underset{h\rightarrow 0^{+}}{\lim }\frac{1}{h}\mathcal{H}%
(G(t),G(t+h)-hG^{\prime }(t))=0
\end{eqnarray*}%
and%
\begin{eqnarray*}
&&\underset{h\rightarrow 0^{+}}{\lim }\frac{1}{h}\mathcal{H}((F\ominus
G)(t),(F\ominus G)(t-h)+h(F^{\prime }+(-G^{\prime }))(t)) \\
&=&\underset{h\rightarrow 0^{+}}{\lim }\frac{1}{h}\mathcal{H}(F(t)\ominus
G(t),F(t-h)\ominus G(t-h)+hF^{\prime }(t)-hG^{\prime }(t)) \\
&\leq &\underset{h\rightarrow 0^{+}}{\lim }\frac{1}{h}\mathcal{H}%
(F(t-h),F(t)+hF^{\prime }(t)) \\
&&+\underset{h\rightarrow 0^{+}}{\lim }\frac{1}{h}\mathcal{H}%
(G(t-h),G(t)-hG^{\prime }(t))=0.
\end{eqnarray*}%
It follows that $F\ominus G\in \mathcal{H}^{1}$ and $(F\ominus G)^{\prime
}=F^{\prime }+(-G^{\prime })$. Now, we suppose that $F\in \mathcal{H}^{2}$
and $G\in \mathcal{H}^{1}$. Then using Propostion \ref{p-1}, we obtain that%
\begin{eqnarray*}
&&\underset{h\rightarrow 0^{+}}{\lim }\frac{1}{h}\mathcal{H}%
((F+G)(t),(F+G)(t+h)-h(F^{\prime }\ominus (-G^{\prime }))(t)) \\
&=&\underset{h\rightarrow 0^{+}}{\lim }\frac{1}{h}\mathcal{H}%
(F(t)+G(t),F(t+h)+G(t+h)+(-hF^{\prime }(t))\ominus (hG^{\prime }(t))) \\
&\leq &\underset{h\rightarrow 0^{+}}{\lim }\frac{1}{h}\mathcal{H}%
(F(t),F(t+h)-hF^{\prime }(t)) \\
&&+\underset{h\rightarrow 0^{+}}{\lim }\frac{1}{h}\mathcal{H}%
(G(t+h),G(t)+hG^{\prime }(t))=0
\end{eqnarray*}%
and%
\begin{eqnarray*}
&&\underset{h\rightarrow 0^{+}}{\lim }\frac{1}{h}\mathcal{H}%
((F+G)(t-h),(F+G)(t)-h(F^{\prime }\ominus (-G^{\prime }))(t)) \\
&=&\underset{h\rightarrow 0^{+}}{\lim }\frac{1}{h}\mathcal{H}%
(F(t-h)+G(t-h),F(t)+G(t)+(-hF^{\prime }(t))\ominus (hG^{\prime }(t))) \\
&\leq &\underset{h\rightarrow 0^{+}}{\lim }\frac{1}{h}\mathcal{H}%
(F(t-h),F(t)-hF^{\prime }(t)) \\
&&+\underset{h\rightarrow 0^{+}}{\lim }\frac{1}{h}\mathcal{H}%
(G(t),G(t-h)+hG^{\prime }(t))=0
\end{eqnarray*}%
It follows that $F+G\in \mathcal{H}^{2}$ and $(F+G)^{\prime }=F^{\prime
}\ominus (-G^{\prime })$. Finally using Propostion \ref{p-1}, we have that%
\begin{eqnarray*}
&&\underset{h\rightarrow 0^{+}}{\lim }\frac{1}{h}\mathcal{H}((F\ominus
G)(t),(F\ominus G)(t+h)-h(F^{\prime }+(-G^{\prime }))(t)) \\
&=&\underset{h\rightarrow 0^{+}}{\lim }\frac{1}{h}\mathcal{H}(F(t)\ominus
G(t),F(t+h)\ominus G(t+h)+(-hF^{\prime }(t))+hG^{\prime }(t)) \\
&\leq &\underset{h\rightarrow 0^{+}}{\lim }\frac{1}{h}\mathcal{H}%
(F(t),F(t+h)-hF^{\prime }(t)) \\
&&+\underset{h\rightarrow 0^{+}}{\lim }\frac{1}{h}\mathcal{H}%
(G(t+h),G(t)+hG^{\prime }(t))=0
\end{eqnarray*}%
and%
\begin{eqnarray*}
&&\underset{h\rightarrow 0^{+}}{\lim }\frac{1}{h}\mathcal{H}((F\ominus
G)(t-h),(F\ominus G)(t)-h(F^{\prime }+(-G^{\prime }))(t)) \\
&=&\underset{h\rightarrow 0^{+}}{\lim }\frac{1}{h}\mathcal{H}(F(t-h)\ominus
G(t-h),F(t)\ominus G(t)+(-hF^{\prime }(t))+hG^{\prime }(t)) \\
&\leq &\underset{h\rightarrow 0^{+}}{\lim }\frac{1}{h}\mathcal{H}%
(F(t-h),F(t)-hF^{\prime }(t)) \\
&&+\underset{h\rightarrow 0^{+}}{\lim }\frac{1}{h}\mathcal{H}%
(G(t),G(t-h)+hG^{\prime }(t))=0.
\end{eqnarray*}%
It follows that $F\ominus G\in \mathcal{H}^{2}$ and $(F\ominus G)^{\prime
}=F^{\prime }+(-G^{\prime })$.
\end{proof}

\vskip0.3cm

\begin{theorem}
\label{th-11}\textit{\ If }$F:[a,b]\rightarrow \mathcal{K}$ \textit{is left
(right) }$gH$\textit{-differentiable at }$t_{0}\in (a,b]$ \textit{(}$%
t_{0}\in \lbrack a,b)$\textit{), then }$F$\textit{\ is left (right)
differentiable at }$t_{0}\in (a,b]$\textit{\ (}$t_{0}\in \lbrack a,b)$%
\textit{) and }$D_{H}^{-}F(t_{0})=F_{-}^{\prime }(t_{0})$ \textit{(}$%
D_{H}^{+}F(t_{0})=F_{+}^{\prime }(t_{0})$\textit{)}. \vskip0.3cm
\end{theorem}

\begin{proof}
If $F$ is left $gH$-differentiable at $t_{0}\in (a,b]$, then there exist an
element $A=D_{H}^{-}F(t_{0})\in \mathcal{K}$ such that%
\begin{equation*}
\begin{tabular}{l}
$\underset{h\rightarrow 0^{+}}{\lim }\frac{1}{h}\mathcal{H}%
(F(t_{0}),F(t_{0}-h)+hA)=\underset{h\rightarrow 0^{+}}{\lim }\frac{1}{h}%
\mathcal{H}(F(t_{0})\ominus F(t_{0}-h),hA)$ \\ 
\\ 
$=\underset{h\rightarrow 0^{+}}{\lim }\mathcal{H}\left( \frac{1}{h}%
(F(t_{0})\ominus F(t_{0}-h)),A\right) =0.$%
\end{tabular}%
\end{equation*}%
or%
\begin{equation*}
\begin{tabular}{l}
$\underset{h\rightarrow 0^{+}}{\lim }\frac{1}{h}\mathcal{H}%
(F(t_{0}-h),F(t_{0})-hA)=\underset{h\rightarrow 0^{+}}{\lim }\frac{1}{h}%
\mathcal{H}(F(t_{0}-h)\ominus F(t_{0}),(F(t_{0})-hA)\ominus F(t_{0}))$ \\ 
\\ 
$\underset{h\rightarrow 0^{+}}{\lim }\frac{1}{h}\mathcal{H}%
(F(t_{0}-h)\ominus F(t_{0}),-hA)=\underset{h\rightarrow 0^{+}}{\lim }\frac{1%
}{h}\mathcal{H}(-(F(t_{0})\ominus F(t_{0}-h)),-hA)$ \\ 
\\ 
$=\underset{h\rightarrow 0^{+}}{\lim }\mathcal{H}\left( \frac{1}{h}%
(F(t_{0})\ominus F(t_{0}-h)),A\right) =0.$%
\end{tabular}%
\end{equation*}%
It follows that $F$\ is left differentiable at $t_{0}\in \lbrack a,b)$\ and $%
F_{-}^{\prime }(t_{0})=D_{H}^{-}F(t_{0})$. A similar proof establishes the
result if $F$ is right $gH$-differentiable at $t_{0}\in \lbrack a,b)$.\vskip%
0.3cm
\end{proof}

\begin{corollary}
\label{cor-4}\textit{\ If }$F:[a,b]\rightarrow \mathcal{Q}$ \textit{is }$H$%
\textit{-differentiable at }$t_{0}\in \lbrack a,b]$\textit{, then }$F$%
\textit{\ is differentiable at }$t_{0}\in \lbrack a,b]$\textit{\ and }$%
D_{H}F(t_{0})=F^{\prime }(t_{0})$.\vskip0.3cm
\end{corollary}

\begin{remark}
\label{rem-7} The converse of the theorem is not true in general as we will
show in next example.\vskip0.3cm
\end{remark}

\begin{example}
\label{ex-7} Consider the function $F:[a,b]\rightarrow \mathcal{K}$ defined
by $F(t)=(2+\sin t)[-1,1]$, $t\in (0,2\pi )$. Then for any $t\in (0,2\pi )$
and $U=[-1,1]$, we have%
\begin{eqnarray*}
&&\underset{h\rightarrow 0^{+}}{\lim }\frac{1}{h}\mathcal{H}%
(F(t+h),F(t)+h\cos t\cdot U) \\
&=&\underset{h\rightarrow 0^{+}}{\lim }\frac{1}{h}\mathcal{H}((2+\sin
(t+h))U,(2+\sin t)U+(h\cos t)U) \\
&=&\underset{h\rightarrow 0^{+}}{\lim }\frac{1}{h}\mathcal{H}(\sin
(t+h)U,(\sin t+h\cos t)U) \\
&=&\underset{h\rightarrow 0^{+}}{\lim }\frac{1}{h}\left\vert \sin (t+h)-\sin
t-h\cos t\right\vert \mathcal{H}_{\mathcal{K}}(U,\theta )) \\
&=&\underset{h\rightarrow 0^{+}}{\lim }\left\vert \frac{\sin (t+h)-\sin t}{h}%
-\cos t\right\vert =0
\end{eqnarray*}%
and%
\begin{eqnarray*}
&&\underset{h\rightarrow 0^{+}}{\lim }\frac{1}{h}\mathcal{H}%
(F(t),F(t-h)+h\cos t\cdot U) \\
&=&\underset{h\rightarrow 0^{+}}{\lim }\frac{1}{h}\mathcal{H}((2+\sin
t)U,(2+\sin (t-h))U+(h\cos t)U) \\
&=&\underset{h\rightarrow 0^{+}}{\lim }\left\vert \frac{\sin t-\sin (t-h)}{h}%
-\cos t\right\vert =0.
\end{eqnarray*}%
It follows that $F$ is differentiable (in fact, $\mathcal{H}^{2}$%
-differentiable) on $(0,2\pi )$ and $F^{\prime }(t)=(\cos t)U$, $t\in
(0,2\pi )$. On the other hand, $F$ is not right $gH$-differentiable nor left 
$gH$-differentiable on $(0,2\pi )$ since $diam(F(t))=2(2+\sin t)$ changes
its monotonicity on $(0,2\pi )$ (see \cite[Theorem 4.1]{banks}).
\end{example}

\section{\textbf{Riemann integral for interval-valued functions}}

\noindent Let $F:[a,b]\rightarrow $\QTR{cal}{K} be a given function. For
each finite partition $\Delta _{n}=\{t_{0},t_{1},...,t_{n}\}$, $%
a=t_{0}<t_{1}<...<t_{n}=b$, of interval $[a,b]$ and for arbitrary system $%
\xi =(\xi _{1},\xi _{2},...,\xi _{n})$ of intermediate points $\xi _{i}\in
\lbrack t_{i-1},t_{i}]$, $i=1,2,...,n$, we consider Riemann sum%
\begin{equation*}
R_{F}(\Delta _{n},\xi )=\sum_{i=1}^{n}(t_{i}-t_{i-1})F(\xi _{i})\text{ and }%
|\Delta _{n}|:=\underset{1\leq i\leq n}{\max }(t_{i}-t_{i-1}).
\end{equation*}%
We say that the function $F:[a,b]\rightarrow $\QTR{cal}{K} is \emph{Riemann
integrable} on $[a,b]$ if there exists an $A\in $\QTR{cal}{K} such that for
each $\varepsilon >0$, there exists a $\delta >0$ so that if $\Delta _{n}$
is any partition of $[a,b]$ and $\xi $ an arbitrary system of intermediate
points, then%
\begin{equation*}
\mathcal{H}(R_{F}(\Delta _{n},\xi ),A)<\varepsilon .
\end{equation*}%
We write $A=\int_{a}^{b}F(t)dt$. It easy to see that, if $F:[a,b]\rightarrow 
$\QTR{cal}{K} is Riemann integrable on $[a,b]$, then the value of the
integral is unique. We observe that, for each partition $\Delta _{n}$ of $%
[a,b]$ with $|\Delta _{n}|\rightarrow 0$ as $n\rightarrow \infty $,\ we have 
\begin{equation*}
\underset{n\rightarrow \infty }{\lim }R_{F}(\Delta _{n},\xi
)=\int_{a}^{b}F(t)dt.
\end{equation*}%
\medskip 

\begin{theorem}
\textit{If }$F:[a,b]\rightarrow $\QTR{cal}{K} \textit{is a continuous
function, then }$F$\textit{\ is Riemann integrable on }$[a,b]$.\vskip0.3cm
\end{theorem}

\begin{proof}
As in the classical proof, the uniform continuity of $F$\ implies that $%
R_{F}(\Delta _{n},\xi )$ is a Cauchy sequence for all sequences of
partitions which have $|\Delta _{n}|\rightarrow 0$ as $n\rightarrow \infty $%
. Consideration of the interleaved sequence $R_{F}(\Delta _{1},\xi )$, $%
R_{F}(\Delta _{1}^{\prime },\xi )$, $R_{F}(\Delta _{2},\xi )$, $R_{F}(\Delta
_{2}^{\prime },\xi )$,... shows that all sequences $R_{F}(\Delta _{n},\xi )$
and $R_{F}(\Delta _{n}^{\prime },\xi )$ for which $|\Delta _{n}|\rightarrow 0
$ and $|\Delta _{n}^{\prime }|\rightarrow 0$\ as $n\rightarrow \infty $,
have the same limit.
\end{proof}

\vskip0.3cm

\noindent Let $F,G:[a,b]\rightarrow $\QTR{cal}{K} be Riemann integrable on $%
[a,b]$. Then the follwing properties are obviously by passing to the limit
form corresponding relations for Riemann sums.\medskip 

\noindent (a) For each $\alpha ,\beta \in \mathbb{R}$, $\alpha F+\beta G$ is
Riemann integrable on $[a,b]$ and%
\begin{equation}
\dint_{a}^{b}(\alpha F(t)+\beta G(t))dt=\alpha \dint_{a}^{b}F(t)dt+\beta
\dint_{a}^{b}G(t)dt.  \label{ad_int}
\end{equation}%
(b) $F$ is Riemann integrable on each subinterval of $[a,b]$ and%
\begin{equation}
\dint_{a}^{b}F(t)dt=\dint_{a}^{c}F(t)dt+\dint_{c}^{b}F(t)dt\text{, }a\leq
c\leq b.  \label{sub-int}
\end{equation}%
(c) $t\mapsto \mathcal{H}(F(t),G(t))$ is Riemann integrable on $[a,b]$, and%
\begin{equation}
\mathcal{H}\left( \dint_{a}^{b}F(t)dt,\dint_{a}^{b}G(t)dt\right) \leq
\dint_{a}^{b}\mathcal{H}(F(t),G(t))dt.  \label{H-ineq}
\end{equation}%
(d) 
\begin{equation}
\frac{1}{b-a}\dint_{a}^{b}F(t)dt\in \overline{co}\{F(t);t\in \lbrack a,b]\},
\label{conv}
\end{equation}%
where $\overline{co}\mathcal{M}$\ means the closed convex hull of subset $%
\mathcal{M}\subset $\QTR{cal}{K}.\vskip0.3cm

\begin{theorem}
\textit{If }$F:[a,b]\rightarrow $\QTR{cal}{K} \textit{is a continuous
function on }$[a,b]$\textit{, then the function }$G:[a,b]\rightarrow $%
\QTR{cal}{Q}\textit{, defined by}%
\begin{equation}
G(t)=\dint_{a}^{t}F(s)ds,\text{ }t\in \lbrack a,b]\text{,}
\label{G_primitiv}
\end{equation}%
\textit{is }$\mathcal{H}^{1}$-\textit{differentiable on }$[a,b]$\textit{\
and }$G^{\prime }(t)=F(t)$\textit{\ for each }$t\in \lbrack a,b]$\textit{.}%
\vskip0.3cm
\end{theorem}

\begin{proof}
Let $t\in \lbrack a,b]$ and $h>0$ such that $t+h,t-h\in \lbrack a,b]$. Since%
\begin{equation*}
\frac{1}{h}\mathcal{H}(G(t),G(t-h)+hF(t))\leq \frac{1}{h}\dint_{t-h}^{t}%
\mathcal{H}\left( F(s),F(t)\right) ds,
\end{equation*}%
\begin{equation*}
\frac{1}{h}\mathcal{H}(G(t-h),G(t)-hF(t))\leq \frac{1}{h}\dint_{t-h}^{t}%
\mathcal{H}\left( \theta ,F(s)-F(t)\right) ds,
\end{equation*}%
\begin{equation*}
\frac{1}{h}\mathcal{H}(G(t+h),G(t)+hF(t))\leq \frac{1}{h}\dint_{t}^{t+h}%
\mathcal{H}\left( F(s),F(t)\right) ds,
\end{equation*}%
\begin{equation*}
\frac{1}{h}\mathcal{H}(G(t),G(t+h)-hF(t))\leq \frac{1}{h}\dint_{t}^{t+h}%
\mathcal{H}\left( \theta ,F(s)-F(t)\right) ds,
\end{equation*}%
and%
\begin{equation*}
\underset{h\rightarrow 0^{+}}{\lim }\frac{1}{h}\dint_{t}^{t+h}\mathcal{H}%
\left( F(s),F(t)\right) ds=\underset{h\rightarrow 0^{+}}{\lim }\frac{1}{h}%
\dint_{t-h}^{t}\mathcal{H}\left( F(s),F(t)\right) ds=0,
\end{equation*}%
we infer that $G$ is $\mathcal{H}^{1}$-differentiable on $[a,b]$\ and $%
G^{\prime }(t)=F(t)$\ for each $t\in \lbrack a,b]$.
\end{proof}

\vskip0.3cm

\begin{theorem}
\textit{Let }$F:[a,b]\rightarrow $\QTR{cal}{K} \textit{be a differentiable
fuction on }$[a,b]$\textit{\ such that }$F^{\prime }$\textit{\ is continuous
on }$[a,b]$.\textit{\ }

\textit{(a) If }$F$\textit{\ is }$\mathcal{H}^{1}$\textit{-differentiable,\
then}%
\begin{equation}
F(t)=F(a)+\dint_{a}^{t}F^{\prime }(s)ds  \label{ch-1}
\end{equation}%
\textit{for any }$t\in \lbrack a,b]$. 

\textit{(b) If }$F$\textit{\ is }$\mathcal{H}^{2}$\textit{-differentiable,\
then}%
\begin{equation}
F(t)=F(a)\ominus \left( -\dint_{a}^{t}F^{\prime }(s)ds\right)   \label{ch-2}
\end{equation}%
\textit{for any }$t\in \lbrack a,b]$.\vskip0.3cm
\end{theorem}

\begin{proof}
(a) Suppose that $F$ is $\mathcal{H}^{1}$-differentiable on $[a,b]$ and $%
F^{\prime }$\ is continuous on $[a,b]$.\textit{. }If we put $%
G(t):=F(a)+\dint_{a}^{t}F^{\prime }(s)ds$, $t\in \lbrack a,b]$, then $%
G^{\prime }(t)=F^{\prime }(t)$ for all $t\in \lbrack a,b]$. Define $m(t):=%
\mathcal{H}(F(t),G(t))$, $t\in \lbrack a,b]$. Then, we have%
\begin{eqnarray*}
m(t+h)-m(t) &=&\mathcal{H}(F(t+h),G(t+h))-\mathcal{H}(F(t),G(t)) \\
&\leq &\mathcal{H}(F(t+h),F(t)+hF^{\prime }(t))+ \\
&&+\mathcal{H}(F(t)+hF^{\prime }(t),G(t)+hF^{\prime }(t)) \\
&&+\mathcal{H}(G(t)+hF^{\prime }(t),G(t+h))-\mathcal{H}(F(t),G(t)) \\
&=&\mathcal{H}(F(t+h),F(t)+hF^{\prime }(t))+ \\
&&+\mathcal{H}(G(t+h),G(t)+hF^{\prime }(t))\text{,}
\end{eqnarray*}%
and thus 
\begin{eqnarray*}
D^{+}m(t) &=&\underset{h\rightarrow 0^{+}}{\lim \sup }\frac{m(t+h)-m(t)}{h}
\\
&\leq &\underset{h\rightarrow 0^{+}}{\lim }\frac{1}{h}\mathcal{H}%
(F(t+h),F(t)+hF^{\prime }(t)) \\
&&+\underset{h\rightarrow 0^{+}}{\lim }\frac{1}{h}\mathcal{H}%
(G(t+h),G(t)+hF^{\prime }(t))=0.
\end{eqnarray*}%
Therefore, $D^{+}m(t)\leq 0$ for all $t\in \lbrack a,b)$ and so $m$ is a
decreasing function on $[a,b]$. Since $m(a)=0$, it follows that $m(t)\leq
m(a)=0$ for all $t\in \lbrack a,b]$. On the other hand, we have that $%
m(t)\geq 0$ for all $t\in \lbrack a,b]$ and so $m(t)=0$ for all $t\in
\lbrack a,b]$; that is, (\ref{ch-1}). (b) Suppose that $F$ is $\mathcal{H}%
^{2}$-differentiable on $[a,b]$\textit{. }If we put $G(t):=F(a)\ominus
\left( -\dint_{a}^{t}F^{\prime }(s)ds\right) $, $t\in \lbrack a,b]$, then by
Corollary 3 in \cite{vl} it follows that $G^{\prime }(t)=F^{\prime }(t)$ for
all $t\in \lbrack a,b]$. As above we obtain that $m(t)=0$ for all $t\in
\lbrack a,b]$; that is, (\ref{ch-2}).
\end{proof}

\vskip0.3cm

\noindent {\large Conclusion} \textit{This new concept of differentiability
for interval-valued functions avoids the use of generalized Hukuhara
difference. It is well known that the generalized difference Hukuhara }$%
A\ominus B$\textit{\ does not generally exist if }$A$\textit{\ and }$B$%
\textit{\ are compact sets in }$R^{n}$\textit{\ with }$n\geq 2$\textit{\ or
if }$A$\textit{\ and }$B$\textit{\ are fuzzy sets. Therefore, this new
concept of differentiability can be much more efficient in these situations
than the concepts previously known. The extension of these results to fuzzy
functions, as well as their applications to differential equations, will be
developed in a few future works}. \textit{We will also extend this concept
to functions with values {}{}in much more general spaces, namely to
functions with values in quasilinear spaces.}

\end{document}